\documentclass[12pt]{article}
\usepackage{amsmath,amsthm,amssymb,multicol}
\usepackage[margin = 1in]{geometry}
\usepackage{pgfplots}
\usepackage{float}

\usepackage{tikz}
\usetikzlibrary{arrows,chains,matrix,positioning,scopes}
\makeatletter
\tikzset{join/.code=\tikzset{after node path={
\ifx\tikzchainprevious\pgfutil@empty\else(\tikzchainprevious)
edge[every join]#1(\tikzchaincurrent)\fi}}}
\makeatother

\tikzset{>=stealth',every on chain/.append style={join},
         every join/.style={->}}
\tikzstyle{labeled}=[execute at begin node=$\scriptstyle,
   execute at end node=$]
   
\title{Random Coxeter Groups}
\author{Angelica Deibel}   

\newcommand{\aas}{asymptotically almost surely }
\newcommand{\EV}{\mathbb{E}}

\newtheorem{lemma}{Lemma}
\newtheorem{cor}{Corollary}
\newtheorem{thm}{Theorem}
\newtheorem{defn}{Definition}

\begin{document}

\maketitle

\begin{abstract}
Much is known about random right-angled Coxeter groups (i.e., right-angled Coxeter groups whose defining graphs are random graphs under the Erd\"os-R\'enyi model). In this paper, we extend this model to study random general Coxeter groups and give some results about random Coxeter groups, including some information about the homology of the nerve of a random Coxeter group and results about when random Coxeter groups are $\delta$-hyperbolic and when they have the
FC-type property.
\end{abstract}

\section{Introduction}

Let $\Gamma$ be a graph with vertices $v_1,\dots,v_k$ in which each edge is labelled with an integer 
at least 2. The Coxeter group associated to $\Gamma$ is
$W_\Gamma = \langle v_1,\dots,v_k | v_i^2, (v_i,v_j)^{m_{i,j}}\rangle$,
where $m_{i,j}$ is the label of the edge connecting $v_i$ and $v_j$. If $m_{i,j} = 2$ whenever 
$v_i$ and $v_j$ are adjacent in $\Gamma$, then $W_\Gamma$ is a right-angled Coxeter group, 
and usually in this case we just think of $\Gamma$ as a graph without the edge labels. 

A lot of previous work has been done on random right-angled Coxeter groups; i.e., Coxeter groups 
associated to a random graph under the Erd\"os-R\'enyi model. In the Erd\"os-R\'enyi model, a random
graph on $n$ vertices is chosen by independently including each possible edge with probability $p(n)$.
Typically, we are interested in the asymptotic behavior as $n \to \infty$. The study of random graphs
is itself a large field with many results (see, e.g., \cite{bollobas}), and when an interesting group 
property can be identified in the defining graph of a right-angled Coxeter group, 
the results and methods of random
graph theory can be applied to give results about the groups (e.g. \cite{thickrelhyp},\cite{racghyp}). 

In the same way, properties of Coxeter groups which can be identified from the defining graph can 
be studied via random edge-labelled graphs under an extended version of the Erd\"os-R\'enyi model.
In this model, a random edge-labelled graph on $n$ vertices is chosen by independently choosing each
possible edge to appear labelled $m$ with probability $p_m(n)$ such that for all $n$,
$\sum\limits_{m=2}^\infty p_m(n) \leq 1$. We will sometimes denote $\sum\limits_{m=2}^\infty p_m(n)$
by $p(n)$. In this paper, we use this model to give some results about
random Coxeter groups which extend previous work on random right-angled Coxeter groups, as well
as some results about general random Coxeter groups which were not interesting in the 
right-angled case.

One area of interest in the study of random right-angled Coxeter groups has been topological 
properties of the nerve of a random right-angled Coxeter group. The nerve of a Coxeter group
$W_\Gamma$ is the simplicial complex whose vertices are the vertices of the defining graph 
$\Gamma$, and in which a simplex appears if the vertices in the simplex generate a finite 
subgroup of $W_\Gamma$. In \cite{davis}, Davis shows that the cohomology of $W_\Gamma$
with coefficients in $\mathbb{Z}W_\Gamma$ can be computed from the homology of the nerve and
certain subcomplexes of the nerve.
In the right-angled case, the nerve is just the flag complex whose 1-skeleton
is the defining graph of $\Gamma$. In \cite{kahle}, Kahle gives conditions on $p$ under which a 
random flag complex \aas has trivial $H_k$ or non-trivial $H_k$ (among other results). Then in 
\cite{daviskahle}, Davis and Kahle use these results along with additional work to determine the
cohomology of a random right-angled Coxeter group. In particular, they show that random
right-angled Coxeter groups are \aas rational duality groups; i.e., that for a random right-angled Coxeter
group $W_\Gamma$, $H^k(W_\Gamma, \mathbb{Q}W_\Gamma)$ is \aas trivial in all but one 
dimension.
In section \ref{dim}, we show that if $n^k p_2^{k\choose 2} \to \infty$ and $\frac{p-p_2}{p_2}\not\to \infty$,
then the nerve of $W_\Gamma$ \aas has trivial $H_i$ for all $i \geq k+1$. Also, in section
\ref{nontrivhom}, we show that if $n^k p_2^{k\choose 2} \to \infty$ and $np_3 \not\to 0$ and 
$np^{k+1} \to 0$, then the nerve of $W_\Gamma$ \aas has non-trivial $H_k$. This 
partially extends the work of Kahle in \cite{kahle} and is the first step towards understanding the 
cohomology of a random Coxeter group in the general setting.

Another of the various properties that have been studied in random right-angled Coxeter groups is 
hyperbolicity. A group is hyperbolic if its Cayley graph has the property that for some $\delta$, 
all triangles are ``$\delta$-slim''; i.e., the $\delta$-neighborhood of the union of two sides of a 
geodesic triangle in the Cayley graph contains the third side. (This notion of hyperbolicity is different
and not equivalent to a Coxeter group being a hyperbolic reflection group.)
In \cite{racghyp}, Charney and Farber give conditions under which right-angled 
Coxeter groups are \aas hyperbolic or not. In section \ref{hyp}, we give conditions under which
 $W_\Gamma$ is \aas hyperbolic, 
and we also give conditions under which $W_\Gamma$ is \aas not hyperbolic.

The study of random general Coxeter groups also produces some questions which were not 
interesting in the right-angled case. One example of this is the FC-type property. We say a Coxeter
group $W_\Gamma$ is of FC type if every complete subgraph of $\Gamma$ generates a finite
subgroup of $W_\Gamma$. In the right-angled case, this is always true since any complete subgraph
generates a subgroup isomorphic to a product of $\mathbb{Z}/2\mathbb{Z}$'s, 
but general Coxeter groups may
or may not have this property. In section \ref{fc}, we give conditions under which a random Coxeter 
group is \aas of FC type and conditions under which it \aas is not.

\section{Subgraphs of Random Edge-Labelled Graphs}

$\Gamma$ is a random edge-labelled graph where edges are labelled $m$ with probability $p_m(n)$ for every $m$,
and $W_\Gamma$ is the Coxeter group associated to $\Gamma$. For each $m$, $\Gamma_m$ is the subgraph of 
$\Gamma$ whose edges are labelled $m$ in $\Gamma$.

The following proofs rely on computing the expected number of subgraphs of $\Gamma$ which are isomorphic 
to a particular graph $\Gamma'$. To simplify the proofs, we will give the general procedure for computing these
expected values here. First, pick an ordering of the vertices $w_1,\dots,w_k$ of $\Gamma'$, so that we can consider a
$k$-tuple $(v_1,\dots,v_k)$ of vertices in $\Gamma$ to be an instance of $\Gamma'$ if the edge between every
pair $v_i,v_j$ has the same label as the edge between $w_i,w_j$. Then for each $k$-tuple 
$\alpha = (v_1,\dots, v_k)$ of vertices of $\Gamma$, let $X_\alpha$ be the random variable which takes the 
value 1 if $\alpha$ is an instance of $\Gamma'$ and takes the value 0 otherwise, and let 
$X = \sum\limits_{\alpha} X_\alpha$. Then the expected number of subgraphs of $\Gamma$ isomorphic to
$\Gamma'$ is $\frac{1}{b}\EV(X)$, where $b$ is the number of permutations of the vertices of an instance of 
$\Gamma'$ which give another instance of $\Gamma'$ (for any particular $\Gamma'$, this is easy to compute, though we typically won't need that information). 
So, the expected number of subgraphs isomorphic to $\Gamma'$ is:
\begin{align*}
\frac{1}{b}\EV(X) = \frac{1}{b} \sum\limits_{\alpha} \EV(X_\alpha)
= \frac{1}{b}n(n-1)\cdots(n-k) \EV(X_\alpha) \sim \frac{1}{b} n^k \EV(X_\alpha),
\end{align*}
(since $\EV(X_\alpha)$ does not actually depend on $\alpha$).
Now, $X_\alpha$ will take the value 1 precisely when each edge between vertices in $\alpha$ has the 
``correct'' label -- i.e., the one that matches the label on the corresponding edge in $\Gamma'$. Since the choices of 
labels for different edges are independent, if edge $e_i$ in $\Gamma'$ is labelled $m_i$, $\EV(X_\alpha)$ is given 
by the product $\prod\limits_i p_{m_i}(n)$. So,
\begin{align*}
\frac{1}{b} \EV(X) \sim \frac{1}{b} n^k \prod\limits_i p_{m_i}(n).
\end{align*}
For a particular graph $\Gamma'$, we can find the expected number of subgraphs
isomorphic to $\Gamma'$ just by knowing (1) the number $b$ of symmetries of $\Gamma'$, (2) the number 
$k$ of vertices in $\Gamma'$, and (3) the number of edges in $\Gamma'$ which are labelled with each label $m$,
all of which are easy to identify.

It will also be important to be able to compute the expected number of subgraphs of $\Gamma$ which are 
isomorphic to two copies of $\Gamma'$ (possibly sharing some vertices and edges). Such subgraphs can be 
characterized by the number of vertices and the number and labels of edges shared by the two copies of $\Gamma$.
We can consider a $(2k - \ell)$-tuple $(s_1,\dots,s_\ell,v_1,\dots,v_{k-\ell},w_1,\dots,w_{k-\ell})$ of vertices in 
$\Gamma$ to be an instance of two copies of $\Gamma'$ sharing $\ell$ vertices if there is some choice of 
$1 \leq i_1 < i_2 < \cdots < i_\ell \leq k$ and $1 \leq j_1 < j_2 < \cdots < j_\ell \leq k$ such that 
the $k$-tuple whose $i_q$'th entry is $s_q$ and whose $r$'th entry for $r \neq i_q$ is $v_{r-\#\{q: i_q < r\}}$
is an instance of $\Gamma'$, and also similarly the $k$-tuple whose $j_q$'th entry is $s_q$ and whose
$r$'th entry for $r \neq j_q$ is $w_{r-\#\{q : j_q < r\}}$ is an instance of $\Gamma'$. 

For each $(2k-\ell)$-tuple $\alpha$, let $Y_{\ell, \alpha}$ be the random variable that takes the value 1 if $\alpha$ is an 
instance of two copies of $\Gamma'$ sharing $\ell$ vertices and takes the value 0 otherwise. 
Then the expected number of pairs of subgraphs isomorphic to $\Gamma'$ which share $\ell$ 
vertices is 

\begin{align*}
\sum\limits_\alpha \EV(Y_{\ell, \alpha}) = 
\sum\limits_{a_2 + a_3 + \cdots + a_\infty = {\ell \choose 2}} b_{\ell, a_2, a_3, \dots, a_\infty} 
n^{k-\ell} p_2^{c_2 - a_2} p_3^{c_3 - a_3} \cdots p_{\infty}^{c_\infty - a_\infty},
\end{align*}

where $c_i$ is the number of edges of $\Gamma'$ labelled $i$, and 
 $b_{\ell, a_2, a_3, \dots, a_\infty}$ is the number of choices of two subgraphs of $\Gamma'$,
each of which has $\ell$ vertices and $a_i$ edges labelled $i$ for each $i$. (I.e., the number of 
choices for the overlapping part of the two copies of $\Gamma'$). Note that for any particular choice
of $\Gamma'$ and $\ell$, only finitely many of the $a_i$ will be non-zero; in the proofs to follow, we 
will simplify notation a bit by dropping those which are zero.

So, the expected number of subgraphs of $\Gamma$ isomorphic to two copies of $\Gamma'$ is
\begin{align*}
\EV(X^2) &= \sum\limits_{\ell = 0}^k \sum\limits_{(2k-\ell)-\text{tuples } \alpha} \EV(Y_{\ell, \alpha})\\
&= \sum\limits_{\ell = 0}^k \sum\limits_{a_2 + a_3 + \cdots + a_\infty = {\ell \choose 2}} b_{\ell, a_2, a_3, \dots, a_\infty} 
n^{k-\ell} p_2^{c_2 - a_2} p_3^{c_3 - a_3} \cdots p_{\infty}^{c_\infty - a_\infty}.
\end{align*}

This looks complicated, but it turns out we won't actually need to know what $b_{\ell, a_2, \dots, a_\infty}$ is 
unless it 
is zero (i.e., unless there is no way to have two copies of $\Gamma'$ share $\ell$ vertices and 
$a_i$ $i$-labelled edges for each $i$), 
except for $b_{0,0, \dots, 0}$ (which will always equal $1$). So, to compute as much of this 
expected value as will be required, we only need to know the number of vertices of $\Gamma'$, the 
number of edges of $\Gamma'$ with each label, and, for each choice of $(\ell, a_2, \dots, a_\infty)$,
whether or not there is a subgraph of $\Gamma'$ with $\ell$ vertices and $a_i$ $i$-labelled 
edges for each $i$. Generally, $\Gamma'$ won't be too complicated and this will be pretty simple.



\section{The nerve of a Coxeter group} \label{nerve}

The nerve of a Coxeter group is a simplicial complex associated to the group whose homology 
gives some information about the cohomology of the group. It is defined 
as follows:

\begin{defn}
 For an edge-labelled graph $\Gamma$, the nerve of the Coxeter groups $W_\Gamma$ is the 
simplicial complex whose vertices are the vertices of $\Gamma$, and whose simplices are the sets
of vertices which generate a finite subgroup of $W_\Gamma$. We will denote the nerve by $N(\Gamma)$.
\end{defn}

It's easy to tell which simplices are contained in $N(\Gamma)$ because the subgroup of 
$W_\Gamma$ generated by any subset of the vertices is also a Coxeter group and the finite
Coxeter groups are well-understood and classified by Dynkin diagrams (see, e.g. \cite{davisbook}).
For reference, the diagrams for the irreducible finite Coxeter groups are given in Figure
 \ref{fig:finitetable} in the Appendix.
 
 The nerve of a right-angled Coxeter group is just the flag complex on its defining graph (since 
 every clique in the defining graph of a right-angled Coxeter group generates a product of 
 $\mathbb{Z}/2\mathbb{Z}$'s, which is finite). 
  In \cite{kahle},
Kahle obtains the following result about the homology of a random flag complex:

\begin{thm} [Kahle] For $k \geq 0$,
\begin{itemize}
	\item If $p^kn \to \infty$ and $p^{k+1}n \to 0$, then the random flag complex \aas has nontrivial
	$H_k$.
	\item If $p^kn \to 0$ or if $p^{2k+1}n \to \infty$, then the random flag complex \aas has trivial 
	$H_k$.
\end{itemize}
\end{thm}

In the next two sections, we will study the nerve of a random Coxeter group in the general setting.
In section \ref{dim}, we give some results about the dimension of the nerve of a random Coxeter 
group and draw conclusions about conditions on the $p_m$ under which $N(\Gamma)$ \aas has 
trivial $H_k$. In section \ref{nontrivhom} we give some conditions on the $p_m$ under which
$N(\Gamma)$ \aas has non-trivial homology.

\section{Dimension of the nerve} \label{dim}

The dimension of the nerve of $\Gamma$ is equal to the size of the largest induced subgraph of 
$\Gamma$ corresponding to a finite subgroup of $W_\Gamma$. \\

For this section, let $p_B = \sum\limits_{m=3}^\infty p_m$.

\begin{thm} 
\begin{itemize}
\item[(i)] If $n^k p_2^{{k \choose 2}} \to 0$ and $\frac{p_B}{p_2} \not\to \infty$, then \aas $\dim(N(\Gamma)) < k$.
\item[(ii)] If $n^k p_2^{{k \choose 2}} \to \infty$, then \aas $\dim(N(\Gamma)) \geq k$.
\item[(iii)] If $n^k p_2^{k\choose 2} \not\to 0$ and $\frac{p_m}{p_2} \to \infty$ for some $m$, then
\aas $\dim(N(\Gamma)) \geq k$.
\end{itemize}
\label{thm:dim}
\end{thm}

\begin{proof}
To prove (i), we will show that $\Gamma$ \aas does not contain any induced subgraph on at least
$k$ vertices which corresponds to a finite subgroup of $W_\Gamma$. Any such subgraph has at 
most $k-1$ edges labelled with a number other than $2$, so it suffices to show that $\Gamma$
\aas does not contain any subgraph on $k$ vertices with fewer than $k$ non-2 edges. But the 
expected number of such subgraphs is 
\begin{align*}
\sum\limits_{i=0}^{k-1} {n\choose k} c_i p_B^i p_2^{{k\choose 2}-i}
\sim \sum\limits_{i=0}^{k-1} c_i n^k p_B^i p_2^{{k\choose 2 }-i}
= \sum\limits_{i=0}^{k-1} c_i n^k p_2^{k\choose 2}\left(\frac{p_B}{p_2}\right)^i \to 0,
\end{align*}
since $n^k p_2^{k\choose 2} \to 0$ and $\frac{p_B}{p_2} \not\to\infty$. 
($c_i$ is the sum over all graphs $\Gamma'$ with $k$ vertices and $i$ edges of the number of 
symmetries of $\Gamma'$) 
So, every induced subgraph of $\Gamma$ which induces a finite subgroup of $W_\Gamma$
is on fewer than $k$ vertices; hence $\dim(N(\Gamma)) < k$

For (ii), we'll show that $\Gamma$ \aas contains a $k$-clique whose edges are all 2's. 
Let $\Gamma_2$ be the (unlabelled) graph whose vertices 
are the vertices of $\Gamma$ and whose
edges are those which are labelled 2 in $\Gamma$. Then $\Gamma_2$ can be interpreted as an 
Erd\"os-R\'enyi random graph with edge probability $p(n) = p_2(n)$. Also, a set of vertices in 
$\Gamma$ spans a $k$-clique whose edges are all 2's if and only if the corresponding subgraph
in $\Gamma_2$ is a $k$-clique, so it suffices to show that a random graph with $n^kp^{k\choose 2}$
\aas contains a $k$-clique. This is a standard result (see \cite{bollobas}).

For (iii), we'll show that $\Gamma$ \aas contains a $k$-clique with one edge labelled $m$ and the
rest labelled $2$. For a $k$-tuple of vertices $\alpha = (v_1,\dots, v_k)$, 
let $X_\alpha$ be the random variable which takes the value 1 if edge $(v_1,v_2)$ is labelled $m$
and all other edges are labelled $2$ and takes the value 0 otherwise, and let 
$X = \sum\limits_\alpha X_\alpha$. Then 
\begin{align*}
\frac{\EV(X^2)}{\EV(X)^2} = 1 + \frac{b_{1,0}}{n} 
	+ \sum\limits_{j=1}^k \left[ \frac{b_{j,1}}{n^j p_m p_2^{{j\choose 2} - 1}} 
							+ \frac{b_{j,0}}{n^j p_2^{{j\choose 2}}}\right],
\end{align*}
where $b_{j,i}$ is the number of ways two $k$-tuples of vertices can share $i$ $m$-edges and 
${j\choose 2} - i$ $2$-edges. 
Since $n^k p_2^{k\choose 2} \not\to 0$, we have $n^j p_2^{j\choose 2} \to \infty$ for $j < k$, and since
additionally $\frac{p_m}{p_2} \to \infty$, we have
$n^j p_m p_2^{{j\choose 2} - 1}\left(\frac{p_m}{p_2}\right) \to \infty$;
hence, $\frac{\EV(X^2)}{\EV(X)^2} \to 1$. So, $\Gamma$ \aas contains a graph on $k$ vertices 
corresponding to a finite subgroup of $\Gamma$; hence $N(\Gamma)$ \aas contains a $k$-simplex
and $\dim(N(\Gamma)) \geq k$.
\end{proof}

\begin{thm}
 If $n^k p_2^{{k \choose 2}} \to \infty$ and $n^{k+1} p_2^{{k+1 \choose 2}} \to 0$, 
and $\frac{p_B}{p_2} \not\to \infty$,
then \aas
$\dim(N(\Gamma)) = k$
\end{thm}

\begin{proof}
Since $n^{k+1} p_2^{k+1 \choose 2} \to \infty$ and $\frac{p_B}{p_2} \not\to \infty$, \aas
$\dim(N(\Gamma)) \leq k$ by part (i) of Theorem \ref{thm:dim}. But since $n^k p_2^{k\choose 2}
\to\infty$, $\dim(N(\Gamma)) \geq k$ by part (ii) of Theorem \ref{thm:dim}. So, $\dim(N(\Gamma)) = k$.
\end{proof}

\begin{cor}
 If $n^k p_2^{{k \choose 2}} \to \infty$ 
and $\frac{p_B}{p_2} \not\to \infty$,
then the nerve of $W_\Gamma$ \aas
has trivial $H_i$ for all $i \geq k+1$.
\end{cor}

\begin{proof}
Since $N(\Gamma)$ has dimension at most $k$, it has trivial $H_i$ for $i \geq k+1$.
\end{proof}

\section{Non-trivial homology of the nerve} \label{nontrivhom}

For $k \geq 1$, denote by $Z_k$ the edge-labelled graph on the $2k+2$ vertices 
$\{x_0^+, \dots, x_k^+, x_0^-, \dots, x_k^-\}$ in which the edges $(x_i^+, x_{i+1}^+)$, $(x_i^-, x_{i+1}^-)$,
$(x_k^+, x_0^-)$, and $(x_k^-, x_0^+)$ are labelled 3, the edges $(x_i^+, x_i^-)$ are labelled $\infty$, and 
all other edges are labelled 2 (in the picture, solid lines indicate edges labelled 3, 
dashed lines indicate edges labelled $\infty$, and edges not shown are labelled 2): 

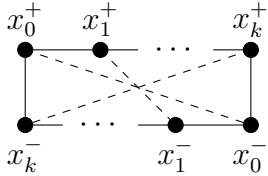
\begin{tikzpicture}

\draw [fill] (0,0) circle(.1);
\draw [fill] (1,0) circle(.1);
\node at (2,0) {$\cdots$};
\draw [fill] (3,0) circle(.1);
\draw [fill] (3,-1) circle(.1);
\draw [fill] (2,-1) circle(.1);
\node at (1,-1) {$\cdots$};
\draw [fill] (0,-1) circle(.1);

\draw (0,0) -- (1.5,0);
\draw (2.5,0) -- (3,0);
\draw (3,0) -- (3,-1);
\draw (3,-1) -- (1.5,-1);
\draw (0.5,-1) -- (0,-1);
\draw (0,-1) -- (0,0);

\draw [dashed] (0,0) -- (3,-1);
\draw [dashed] (1,0) -- (2,-1);
\draw [dashed] (3,0) -- (0,-1);

\node at (0,.4) {$x_0^+$};
\node at (1,.4) {$x_1^+$};
\node at (3,.4) {$x_k^+$};
\node at (3,-1.4) {$x_0^-$};
\node at (2,-1.4) {$x_1^-$};
\node at (0,-1.4) {$x_k^-$};

\end{tikzpicture}

\begin{lemma}
The nerve of $Z_k$ has non-trivial $H_{k}$.
\end{lemma}

\begin{proof}
Any set of $k+1$ vertices of $Z_k$ gives a simplex in $Z_k$ iff it does not contain any pair $x_i^+$,
$x_i^-$, so each $k$-simplex of $N(Z_k)$ is given by $[x_0^\pm, \dots, x_k^\pm]$ for some choice of
$\pm$ for each $i$. For each map $\sigma: \{0,\dots,k\} \to \{\pm\}$, let $d_\sigma$ be the number 
of things that map to $-$. Consider
$\sum\limits_{\sigma} (-1)^{d_\sigma} [x_0^{\sigma(0)}, \dots, x_k^{\sigma(k)}]$.
In the image under $\partial_k$, each $[x_0^{\sigma(0)},\dots,\hat{x_i^{\sigma(i)}}\dots, x_k^{\sigma(k)}]$ will appear
twice -- once from the term from which $x_i^+$ was removed and once from the term from which $x_i^-$ was removed,
and these two terms will have opposite sign (since the term in the pre-image which contained $x_i^-$
had one more negative sign than the term which contained $x_i^+$). So, 
$\sum\limits_{\sigma} (-1)^{d_\sigma} [x_0^{\sigma(0)}, \dots, x_k^{\sigma(k)}] \in \ker(\partial_k) = H_k(N(Z_k))$ (since $N(Z_k)$ has no $k+1$-simplices). 
\end{proof}

\begin{lemma}
If an edge-labelled graph $\Gamma$ contains a subgraph isomorphic to $Z_k$ in which the 
$\{x_i^+\}$ do not have a common neighbor in $\Gamma$, then the nerve of $\Gamma$ retracts 
onto the nerve of $Z_k$.
\label{lemma:retract}
\end{lemma}

\begin{proof}
Define the retraction $r$ on the vertices as follows: first, send each of the $x_i^\pm$ to itself. Then, for any other vertex
$v$, there will be some $x_i^+$ which is not adjacent to $v$. Send $v$ to $x_i^-$. 
To show that $r$ extends simplicially to a retraction $N(\Gamma) \to N(Z_k)$, it suffices to show that 
whenever
$[v_{i_1},\dots,v_{i_j}]$ is a simplex in $N(\Gamma)$, $[r(v_{i_1}),\dots, r(v_{i_j})]$ is also a simplex in 
$N(Z_k)$. This is equivalent to showing that 
 $\{r(v_{i_1}), \dots, r(v_{i_j})\}$ induce a complete subgraph in $Z_k$, since every complete subgraph
 of $Z_k$ generates a finite subgroup of $W_\Gamma$.
But suppose this is not the case; i.e., suppose there exist $a$ and $b$ such that 
$r(v_{i_a})$ and $r(v_{i_b})$ are not adjacent in $Z_k$. Then $\{r(v_{i_a}),r(v_{i_b})\} = \{x_i^+,x_i^-\}$
for some $i$; without loss of generality, say $r(v_{i_a}) = x_i^+$ and $r(v_{i_b}) = x_i^-$. Then 
$v_{i_a} = x_i^+$, and $v_{i_b}$ is not adjacent to $x_i^+$, a contradiction since 
$\{v_{i_1},\dots,v_{i_j}\}$ induces a complete subgraph in $\Gamma$.
So, $r: N(\Gamma) \to N(Z_k)$ is a retraction.
\end{proof}

\begin{lemma}
Let $\Gamma$ be a random edge-labelled graph whose edges are labelled $m$ with probability $p_m(n)$. 
Then if $n^kp_2^{k\choose 2} \to \infty$ and $np_3 \not\to 0$ and $n(1-p_\infty)^{k+1} \to 0$, 
$\Gamma$ \aas contains a subgraph isomorphic to $Z_k$ in which 
the $\{x_i^+\}$ do not have a common
neighbor in $\Gamma$.
\label{lemma:randomZk}
\end{lemma}

\begin{proof}
First, we'll show that $\Gamma$ \aas contains a subgraph isomorphic to $Z_k$. For each $k$-tuple
$\alpha$, let $X_\alpha$ be the random variable which takes the value 1 if $\alpha$ is an instance of
$Z_k$ and 0 otherwise. Let $X_k = \sum\limits_\alpha X_\alpha$. Then
\begin{align*}
\frac{\EV(X_k^2)}{\EV(X_k)^2} = 1 + \sum\limits_{i=1}^{2k+2} \sum\limits_{j=0}^{i/2} 
	\sum\limits_{\ell = 0}^i \frac{b_{i,{i\choose 2}-j-\ell,j,\ell}}{ n^i p_2^{{i\choose 2}-j-\ell} p_3^j p_\infty^\ell},
\end{align*}
where $b_{i,{i\choose 2}-j-\ell,j,\ell}$ is the number of ways two $(2k+2)$-tuples can share $i$ 
vertices, ${i\choose 2}-j-\ell$ 2-edges, $j$ 3-edges, and $\ell$ $\infty$-edges.
Since $n^k p_2^{k\choose 2} \to \infty$, also $n^i p_2^{i\choose 2} \to \infty$ for $i<k$, and since 
also $\frac{p_3}{p_2} \not\to 0$ and $n(1-p_\infty)^{k+1} \to 0 \implies p_\infty \not\to 0$, for each $i$
$n^i p_2^{{i\choose 2} - j - \ell} p_3^\ell p_\infty^j 
= n^i p_2^{{i\choose 2 } - j} \left(\frac{p_3}{p_2}\right)^\ell p_\infty^j \to \infty$. So,
$\frac{\EV(X_k^2)}{\EV(X_k)^2} \to 1$, and hence $\Gamma$ \aas contains a subgraph isomorphic to 
$Z_k$. \\
Now, the probability that a particular instance of $Z_k$ is such that the $\{x_i^+\}$ have a common 
neighbor is $n(1-p_\infty)^{k+1} \to 0$; hence the probability that in a particular instance of $Z_k$
the $\{x_i^+\}$ have no common neighbor $\to 1$.
So, the probability that $\Gamma$ contains a $Z_k$ with no
common neighbor for $\{x_i^+\}$ is 
\begin{align*}
P(\exists Z_k \text{ with no common neighbor})
&= P(\exists Z_k)P(\exists Z_k \text{ with no common neighbor}|\exists Z_k)\\
&\geq P(\exists Z_k)P(\text{a particular } Z_k \text{ has no common neighbor}) \to 1.
\end{align*}
\end{proof}

\begin{thm}
Let $\Gamma$ be a random edge-labelled graph whose edges are labelled $m$ with probability 
$p_m(n)$.
Then if $n^kp_2^{k\choose 2} \to \infty$ and $np_3 \not\to 0$ and $n(1-p_\infty)^{k+1} \to 0$, 
the nerve of $W_\Gamma$ \aas has non-trivial $H_k$.
\end{thm}
\begin{proof}
By Lemma \ref{lemma:randomZk}, $\Gamma$ \aas contains a 
subgraph which is an instance of $Z_k$ such that 
the $\{x_i^+\}$ have no common neighbor. So, by Lemma \ref{lemma:retract}, 
$N(\Gamma)$ retracts onto $N(Z_k)$,
which has nontrivial $H_k$.
\end{proof}

\section{Hyperbolicity} \label{hyp}

Another property which is of interest for Coxeter groups is hyperbolicity, defined as follows:

\begin{defn}
\begin{itemize}
	\item For $\delta \geq 0$, a metric space $X$ is $\delta$-hyperbolic if every geodesic triangle in
	$X$ is ``$\delta$-slim''; i.e., if the $\delta$-neighborhood of the union of two sides of the triangle
	contains the third side. 
	\item A metric space $X$ is hyperbolic if it is $\delta$-hyperbolic for some $\delta$.
	\item A group $G$ is hyperbolic if its Cayley graph with respect to some (and hence any)
	finite generating set is a hyperbolic metric space.
\end{itemize}
\end{defn}

In \cite{moussong}, Moussong gives the following condition for when a 
Coxeter group is hyperbolic in this sense:

\begin{thm}[Moussong] 
A Coxeter group $W_\Gamma$ is hyperbolic iff the following hold:
\begin{itemize}
	\item $\Gamma$ does not contain a subgraph on 3 or more vertices which generates a  
	Euclidean reflection group in $W_\Gamma$
	\item $\Gamma$ does not contain two disjoint subgraphs $S$ and $T$ such that each of
	$S$ and $T$ generates an infinite subgroup of $\Gamma$ and every pair of vertices 
	$s \in S$ and $t \in T$ is adjacent via an edge labelled 2.
\end{itemize}
\end{thm}

These conditions are easy to check because the finite Coxeter groups and Euclidean reflection 
groups are well understood. The tables of both are given in the Appendix for reference 
(the finite Coxeter groups in Table \ref{fig:finitetable} and the Euclidean reflection groups in Table
\ref{fig:affinetable}).

\begin{lemma}
 If $n^4 p_2^4 p_\infty^2 \to \infty$, then $\Gamma$ \aas contains the following subgraph:
\begin{tikzpicture}
\draw[fill] (0,0) circle [radius = 0.05];
\draw[fill] (1,0) circle [radius = 0.05];
\draw[fill] (0,1) circle [radius = 0.05];
\draw[fill] (1,1) circle [radius = 0.05];
\draw (0,0) -- (1,0);
\draw (1,0) -- (1,1);
\draw (1,1) -- (0,1);
\draw (0,1) -- (0,0);
\draw [dashed] (0,0) -- (1,1);
\draw [dashed] (0,1) -- (1,0);
\end{tikzpicture}, where the solid edges are labelled 2 and the dashed edges are labeled $\infty$.
\label{lemma:emptysquare}
\end{lemma}

\begin{proof}
If $X$ is the number of such subgraphs, we have:
\begin{align*}
\frac{\EV(X^2)}{\EV(X)^2} = 1 + \frac{b_{1,0}}{n} + \frac{b_{2,1}}{n^2p_2} + \frac{b_{2,0}}{n^2p_\infty}
+ \frac{b_{3,2}}{n^3 p_2^2 p_\infty} + \frac{b_{4,4}}{n^4 p_2^4 p_\infty^2},
\end{align*}
where $b_{i,j}$ is the number of ways two instances of the subgraph in question can share 
$i$ vertices and $j$ 2-edges.
To show this $\to 1$, we need to show that the denominator of each term besides the first 
$\to \infty$. Since $n^4p_2^4 p_\infty^2 \to \infty$:
\begin{itemize}
\item $n^2 p_2 \to \infty$ since $n^4 p_2^4 p_\infty^2 = (n^2p_2)^2 (p_2^2 p_\infty^2)$
\item $n^2 p_\infty \to \infty$ since $n^4 p_2^4 p_\infty^2 = (n^2 p_\infty)^2 (p_2^4)$
\item $n^3 p_2^2 p_\infty \to \infty$ since 
	$n^4 p_2^4 p_\infty^2 = (n^3 p_2^2 p_\infty)(n p_2^2 p_\infty)$.
\end{itemize}
\end{proof}

\begin{thm}
 If $n^4 p_2^4 p_\infty^2 \to \infty$, then $W_\Gamma$ is \aas not hyperbolic.
\end{thm}

\begin{proof}
By Lemma \ref{lemma:emptysquare}, 
$\Gamma$ \aas has a subgraph isomorphic to a square with edges labelled 2 whose 
diagonals are labeled $\infty$. This corresponds to a direct product of two infinite subgroups, which 
by Moussong's condition means that $W_\Gamma$ is \aas not hyperbolic.
\end{proof}

\begin{thm}
 If $np_3 \to \infty$, then $W_\Gamma$ is \aas not hyperbolic.
\end{thm}

\begin{proof}
If $np_3 \to \infty$, then $\Gamma$ \aas contains a 3-labelled triangle: Let $\Gamma_3$ be the
(unlabelled) graph whose vertices are the vertices of $\Gamma$ and whose edges are those which 
are labelled 3 in $\Gamma$. Then $\Gamma_3$ can be interpreted as an Erd\"os-R\'enyi random
graph with edge probability $p(n) = p_3(n)$, and $\Gamma$ contains a 3-labelled triangle if and 
only if $\Gamma_3$ contains a triangle. So, it suffices to show that a random graph with 
$np \to \infty$ contains a triangle, which is a standard result (see \cite{bollobas}).
By Moussong's condition, this means
$W_\Gamma$ is \aas not hyperbolic.
\end{proof}

\begin{thm}
If $np_4 \to \infty$ and $np_2 \not\to 0$, then $W_\Gamma$ is \aas not hyperbolic.
\end{thm}
\begin{proof}
For each 3-tuple of vertices $\alpha = (v_1,v_2,v_3)$, let $X_\alpha$ be the random variable which
takes the value 1 if edges $(v_1,v_2)$ and $(v_2,v_3)$ are labelled 4 and edge $(v_1,v_3)$ is 
labelled 2. Let $X = \sum\limits_\alpha$. Then
\begin{align*}
\frac{\EV(X^2)}{\EV(X)^2} = 1 + \frac{b_{1,0}}{n} + \frac{b_{2,0}}{n^2p_2} + \frac{b_{2,1}}{n^2p_4}
+ \frac{b_{3,2}}{n^3p_4^2p_2},
\end{align*}
where $b_{i,j}$ is the number of ways two $3$-tuples can share $i$ vertices and $j$ 4-edges.
This $\to 1$ since $n \to \infty$, $n^2p_2 = n(np_2) \to \infty$, $n^2p_4 \to \infty$, and 
$n^3p_4^2p_2 = (np_4)^2(np_2)$. So, $\Gamma$ \aas contains a triangle whose edges are
labelled 4, 4, 2; hence is not hyperbolic by Moussong's condition.
\end{proof}

\begin{thm}
If $np_6 \to \infty$ and $np_3 \not\to 0$ and $np_2 \not\to 0$, then $W_\Gamma$ is \aas not 
hyperbolic.
\end{thm}
\begin{proof}
For each 3-tuple of vertices $\alpha = (v_1,v_2,v_3)$, let $X_\alpha$ be the random variable which
takes the value 1 if edge $(v_1,v_2)$ is labelled 6, edge $(v_2,v_3)$ is labelled 3, and edge 
$(v_1,v_3)$ is labelled 2. Let $X = \sum\limits_{\alpha} X_\alpha$. Then
\begin{align*}
\frac{\EV(X^2)}{\EV(X)^2} = 1 + \frac{b_{1,0,0}}{n} + \frac{b_{2,1,0}}{n^2p_6} + \frac{b_{1,0,1}}{n^2p_3}
+ \frac{b_{2,0,0}}{n^2p_2} + \frac{b_{3,1,1}}{n^3 p_6p_3p_2},
\end{align*}
where $b_{v,e_6,e_3}$ is the number of ways two 3-tuples can share $v$ vertices, $e_6$ 6-edges, 
and $e_3$ 3-edges. This $\to 1$ since $n \to \infty$, $n^2p_6 \to \infty$, $n^2p_3 \to \infty$, 
$n^2 p_2 \to \infty$, and $n^3 p_6p_3p_2 = (np_6)(np_3)(np_2) \to \infty$. So, $\Gamma$ \aas 
contains a triangle whose edges are labelled 6, 3, 2; hence is not hyperbolic by Moussong's condition.
\end{proof}

\begin{lemma} If $np_3 \to 0$, $np_2 \not\to \infty$, $np_4 \not\to \infty$, and $np_6 \not\to \infty$,
and at least one of $np_2$ or $np_4\to 0$, 
then $W_\Gamma$ \aas contains no subgraph on at least three vertices corresponding to
a Euclidean reflection group.
\label{lemma:noaffine}
\end{lemma}
\begin{proof}
We will show that each of the Euclidean reflection groups \aas does not appear in $\Gamma$:
\begin{itemize}
\item The expected number of $\tilde{A}_{k-1}$ is $n^k p_3^k p_2^{{k\choose 2} - k}$, so the total number
over all $k$ is 
\begin{align*}
\sum\limits_{k=3}^\infty n^k p_3^k p_2^{{k \choose 2} - k} \leq \sum\limits_{k=3}^\infty (np_3)^k
= \frac{np_3}{1-np_3} - np_3 - (np_3)^2 \to 0
\end{align*}
since $np_3 \to 0$.
\item The expected number of $\tilde{B}_{k-1}$ is $n^k p_4 p_3^{k-2} p_2^{{k\choose 2} - k + 1}$, so the
total number over all $k$ is 
\begin{align*}
\sum\limits_{k=4}^\infty n^k p_4 p_3^{k-2} p_2^{{k\choose 2} - k + 1} 
\leq (np_2)^2 \sum\limits_{k=4}^\infty (np_3)^{k-2}
= (np_2)^2 \left(\frac{np_3}{1-np_3} - np_3\right) \to 0
\end{align*}
since $np_3 \to 0$ and $np_2 \not\to \infty$.
\item The expected number of $\tilde{C}_{k-1}$ is $n^k p_4^2 p_3^{k-2} p_2^{{k\choose 2} - k + 1}$, so 
the total number over all $k$ is 
\begin{align*}
\sum\limits_{k=4}^\infty n^k p_4^2 p_3^{k-2} p_2^{{k\choose 2} - k + 1}
\leq (np_2)^3 \sum\limits_{k=4}^\infty (np_3)^k-3 
= (np_2)^3 \left(\frac{np_3}{1-np_3}\right) \to 0
\end{align*}
since $np_3 \to 0$ and $np_2 \not\to \infty$.
\item The expected number of $\tilde{D}_{k-1}$ is $n^k p_3^{k-1} p_2^{{k\choose 2} - k + 1}$, so the total
number over all $k$ is
\begin{align*}
\sum\limits_{k=6}^\infty n^k p_3^{k-1}p_2^{{k\choose 2} - k + 1} 
&\leq (np_2) \sum\limits_{k=6}^\infty (np_3)^{k-1}\\
&= (np_2)\left(\frac{np_3}{1-np_3} - np_3 - (np_3)^2 - (np_3)^3 - (np_3)^4\right) \to 0
\end{align*}
since $np_3 \to 0$ and $np_2 \not\to \infty$.
\item The expected number of $\tilde{B}_2$ is $n^3 p_4^2 p_2 = (np_4)^2(np_2) \to 0$ since 
$np_4 \not\to \infty$ and $np_2 \not\to \infty$ and one of them $\to 0$.
\item The expected number of $\tilde{G}_2$ is $n^3 p_6 p_3 p_2 = (np_6)(np_3)(np_2) \to 0$ since
$np_3 \to 0$, $np_6 \not\to\infty$ and $np_2 \not\to\infty$.
\item The expected number of $\tilde{F}_4$ is $n^5p_4p_3^3p_2^6 \leq (np_3)^3(np_2)^2 \to 0$ since
$np_3 \to 0$ and $np_2 \not\to \infty$.
\item The expected number of $\tilde{E}_{k-1}$ for $k = 7,8,9$ is $n^k p_3^{k-1}p_2^{{k\choose 2} - k + 1} \leq (np_2)(np_3)^{k-1} \to 0$ since $np_3 \to 0$ and 
$np_2 \not\to \infty$.
\end{itemize}
The expected number of subgraphs on at least three vertices corresponding to an irreducible
affine Coxeter
group is the sum of the above, and hence $\to 0$. So, $\Gamma$ \aas contains no such subgraph, 
and hence \aas contains no subgraph on at least three vertices corresponding to any affine 
Coxeter group.
\end{proof}

\begin{lemma}
 If $np_2 \to 0$, then $\Gamma$ \aas does not contain any subgraph corresponding to a 
direct product of infinite parabolic subgroups of $W_\Gamma$.
\label{lemma:noproduct}
\end{lemma}

\begin{proof}
Any infinite parabolic subgroup of $W_\Gamma$ must have at least two generators, so it suffices 
to show $\Gamma$ \aas does not contain any direct product of two edges with any labels $m_1$
and $m_2$. The expected number of such subgraphs is $n^4 p_2^4 p_{m_1} p_{m_2}
\leq (np_2)^4 \to 0$ since $np_2 \to 0$.
\end{proof}

\begin{thm}
 If  $np_3 \to 0$, $np_2 \to 0$, $np_4 \not\to \infty$, and $np_6 \not\to \infty$, then $W_\Gamma$ is \aas hyperbolic.
 \end{thm}
 
\begin{proof}
By Lemmas \ref{lemma:noaffine} and \ref{lemma:noproduct}, 
$W_\Gamma$ \aas satisfies Moussong's conditions; hence is hyperbolic.
\end{proof}

\section{The FC-type property} \label{fc}


\begin{defn}
A Coxeter group $W_\Gamma$ is said to be of FC-type if every clique in $\Gamma$ generates a 
finite subgroup of $W_\Gamma$. 
\end{defn}

Denote by $\Gamma_3$ the (unlabelled) graph whose vertices are the vertices of $\Gamma$ and
whose edges are those edges which are labelled 3 in $\Gamma$. Note that $\Gamma_3$
can be interpreted as an Erd\"os-R\'enyi random graph with edge probability $p(n) = p_3(n)$.
We will say a graph $\Gamma'$ is a 3-labelled cycle if $\Gamma'_3$ is a cycle and 
we will say $\Gamma'$ is a 3-labelled tree if $\Gamma'_3$ is a tree. A connected 
3-labelled subgraph of $\Gamma$ is a subgraph of $\Gamma$ which is connected in $\Gamma_3$.

\begin{lemma} 
If $n^5 p_3^4 \to 0$, then $\Gamma_3$ \aas does not contain a tree on 5 vertices.
\end{lemma}

\begin{proof}
There are four trees on $5$ vertices: 
(a) \begin{tikzpicture}
\draw[fill] (0,0) circle [radius = 0.1];
\draw[fill] (1,0) circle [radius = 0.1];
\draw[fill] (2,0) circle [radius = 0.1];
\draw[fill] (3,0) circle [radius = 0.1];
\draw[fill] (4,0) circle [radius = 0.1];
\draw (0,0) -- (1,0);
\draw (1,0) -- (2,0);
\draw (2,0) -- (3,0);
\draw (3,0) -- (4,0);
\end{tikzpicture},
(b) \begin{tikzpicture}
\draw[fill] (0,0) circle [radius = 0.1];
\draw[fill] (1,0) circle [radius = 0.1];
\draw[fill] (2,0) circle [radius = 0.1];
\draw[fill] (1,1) circle [radius = 0.1];
\draw[fill] (1,2) circle [radius = 0.1];
\draw (0,0) -- (1,0);
\draw (1,0) -- (2,0);
\draw (1,0) -- (1,1);
\draw (1,1) -- (1,2);
\end{tikzpicture},
 (c) \begin{tikzpicture}
\draw[fill] (0,1) circle [radius = 0.1];
\draw[fill] (1,0) circle [radius = 0.1];
\draw[fill] (1,1) circle [radius = 0.1];
\draw[fill] (2,1) circle [radius = 0.1];
\draw[fill] (1,2) circle [radius = 0.1];
\draw (1,0) -- (1,1);
\draw (1,1) -- (0,1);
\draw (1,1) -- (2,1);
\draw (1,1) -- (1,2);
\end{tikzpicture},
and (d) \begin{tikzpicture}
\draw [fill] (0,0) circle(0.1);
\draw [fill] (1,0) circle(0.1);
\draw [fill] (1,1) circle(0.1);
\draw [fill] (2,0) circle(0.1);
\draw [fill] (3,0) circle(0.1);
\draw (0,0) -- (3,0);
\draw (1,0) -- (1,1);
\end{tikzpicture}
The expected number of graphs isomorphic to (a) is $\frac{1}{2}n^5p_3^4$, the expected number of graphs
isomorphic to (b) is $\frac{1}{2} n^5 p_3^4$, the number of graphs isomorphic to (c) is 
$\frac{1}{24} n^5 p_3^4$, and the expected number of graphs isomorphic to (d) is $n^5 p_3^4$. 
So, the total expected number of trees on 5 vertices is 
$\frac{49}{24} n^5 p_3^4 \to 0$.
\end{proof}

\begin{lemma}
 If $np_3 \to 0$, $\Gamma_3$ \aas does not contain any cycle.
\end{lemma}

\begin{proof}
The expected number of $k$-cycles is $\frac{1}{2k}n^kp_3^k$, so the total expected number of cycles is
$\sum\limits_{k=3}^\infty \frac{1}{2k} n^k p_3^k 
= \frac{1}{2} \left(\sum\limits_{k=1}^\infty \frac{(np^3)^k}{k} - np_3 - \frac{(np_3)^2}{2}\right)
= \frac{1}{2} \left(-\ln(1-np_3) - np_3 - \frac{(np_3)^2}{2}\right) \to 0$.
\end{proof}

\begin{lemma}
 If $np_3 \to 0$ and $np_2^6 \not\to \infty$, or if $np_3 \not\to \infty$ and $np_2^6 \to 0$, then
$\Gamma$ \aas does not contain a 3-labelled tree on 5 vertices.
\end{lemma}

\begin{proof}
The expected number of 3-labelled trees on 5 vertices in $\Gamma$ is 
$\frac{49}{24} n^5 p_3^4 p_2^6 = \frac{49}{24}(np_3)^4(np_2^6) \to 0$.
\end{proof}

Let $p_B = \sum\limits_{m=4}^\infty p_m$. A ``$B$-labelled'' edge of $\Gamma$ is one labelled with any number
greater than $3$.

\begin{thm}
 If $n^3 p_B^2 \to 0$, $n^3 p_B p_3 \to 0$, and $n^5p_3^4 \to 0$, then $A_\Gamma$ is \aas of FC-type.
\end{thm}

\begin{proof}
Since $n^5 p_3^4 \to 0$, the connected 3-labelled subgraphs of $\Gamma$ are all associated to finite Coxeter
groups by the previous theorem. The expected number of adjacent pairs of $B$-labelled edges is 
$\frac{1}{2} n^3 p_B^2 \to 0$, and the expected number of 
adjacent pairs of edges with one labelled $B$ and the other labelled $3$ is $\frac{1}{2}n^3 p_B p_3 \to 0$. So, 
$B$-labelled edges \aas do not appear adjacent to any 
edge labelled $m$ with $m \geq 3$. Hence, the only connected $3$-and-$B$-labelled subgraphs which 
appear with positive probability are 3-labelled trees on at most 4 vertices whose other edges in 
$\Gamma$ are all labelled 2, as well as $B$-labelled
edges. Since each of these is associated to a finite Coxeter group, $A_\Gamma$ is \aas of FC-type.
\end{proof}

\begin{thm}
 If $np_B \to 0$ and $np_2 \not\to \infty$ and $np_3 \to 0$, then $A_\Gamma$ is \aas of FC-type.
\end{thm}

\begin{proof}
The expected number of triangles with edges labelled $B,B,B$ is $\frac{1}{6}n^3p_B^3 = (np_B)^3 \to 0$. The expected 
number of triangles with edges labelled $B,B,3$ is $\frac{1}{2} n^3 p_B^2 p_3 = \frac{1}{2}(np_B)^2(np_3) \to 0$.
The expected number of triangles with edges labelled $B,B,2$ is 
$\frac{1}{2} n^3 p_B^2 p_2 = \frac{1}{2}(np_B)^2(np_2) \to 0$. So, $\Gamma$ \aas does not contain any adjacent 
pair of $B$-labelled edges. The expected number of triangles whose edges are 
labelled $B,3,3$ is $\frac{1}{2} n^3 p_B p_3^2 = \frac{1}{2} (np_B)(np_3)^2 \to 0$. The expected number of 
triangles whose edges are labelled $B,3,2$ is $n^3 p_B p_3 p_2 = (np_B)(np_3)(np_2) \to 0$. So, $\Gamma$ \aas does
not contain any adjacent pairs of edges with one labelled $B$ and the other labelled $3$. So, the only connected 
3-and-$B$-labelled subgraphs which appear with positive probability in $\Gamma$ are those listed in the previous 
theorem.
Since each of these corresponds to a finite Coxeter group, $A_\Gamma$ is \aas of FC-type.
\end{proof}

\begin{thm}
 If $np_3 \to \infty$, $A_\Gamma$ is \aas not of FC-type
\end{thm}

\begin{proof}
For every 3-tuple $\alpha = (v_1,v_2,v_3)$ of vertices in $\Gamma$, let $X_\alpha$ be the random variable which
takes the value 1 if $\alpha$ spans a 3-labelled triangle, and takes the value 0 otherwise. Let 
$X = \sum\limits_\alpha X_\alpha$ (so the expected number of 3-labelled triangles in $\Gamma$ is 
$\frac{1}{6} \EV(X)$). Then $\EV(X) = n^3 p_3^3$, so $\EV(X)^2 = n^6p_3^6$. 
$\EV(X^2) = b_{0,0} n^6p_3^6 + b_{1,0} n^5 p_3^6 + b_{2,1} n^4 p_3^5 + b_{3,3} n^3 p_3^3$,
where $b_{i,j}$ is the number of ways two 3-tuples can share $i$ vertices and $j$ 3-edges. 
So, $\frac{\EV(X^2)}{\EV(X)^2} = b_{0,0} + \frac{b_{1,0}}{n} + \frac{b_{2,1}}{n^2p_3} + \frac{b_{3,3}}{n^3p_3^3}$.
This $\to 1$ since $b_{0,0} = 1$, $n \to \infty$, $n^2p_3 = n(np_3) \to \infty$, and $n^3 p_3^3 = (np_3)^3 \to \infty$.
\end{proof}

\begin{thm}
 If for $k = 2$ or $3$, $np_B$ and $np_k$ both $\not\to 0$ and one of them $\to \infty$, then 
$A_\Gamma$ is \aas not of FC-type.
\end{thm}

\begin{proof}
For every 3-tuple $\alpha = (v_1,v_2,v_3)$, let $X_\alpha$ be the random variable which takes the value 1 if 
edges $(v_1,v_2)$ and $(v_2,v_3)$ are labelled $B$ and edge $(v_3,v_1)$ is labelled $k$, and takes the value 0 otherwise. 
Let $X = \sum\limits_{\alpha} X_\alpha$ (so the expected number of triangles with two $B$-labelled edges
and one $k$-labelled edge is $\frac{1}{2} \EV(X)$). Then $\EV(X) = n^3 p_B^2 p_k$, so 
$\EV(X)^2 = n^6 p_B^4 p_k^2$. 
$\EV(X^2) = b_{0,0,0} n^6 p_B^4 p_k^2 + b_{1,0,0} n^5 p_B^4 p_k^2 + b_{2,1,0} n^4 p_B^3 p_k^2
+ b_{2,0,1} n^4 p_B^4 p_k + b_{3,2,1} n^3p_B^2 p_k$,
where $b_{i,j,m}$ is the number of ways two of these triangles can share $i$ vertices, $j$ $B$-edges, and $m$
$k$-edges.
So, $\frac{\EV(X^2)}{\EV(X)^2} = b_{0,0,0} + \frac{b_{1,0,0}}{n} + \frac{b_{2,1,0}}{n^2p_B}
+ \frac{b_{2,0,1}}{n^2p_k} + \frac{b_{3,2,1}}{n^3p_B^2 p_k} \to 1$ since
$b_{0,0,0} = 1$, $n \to \infty$, $n^2p_B = n(np_B) \to \infty$, $n^2 p_k = n(np_k) \to \infty$, 
and $n^3p_B^2 p_k = (np_B)^2(np_k) \to \infty$.
\end{proof}

\section{Appendix: Coxeter diagrams of irreducible finite and Euclidean reflection groups}

In the following tables, we use the Dynkin diagram convention, so unlabelled edges
should be interpreted as 3-labelled in the defining graph and missing edges should be interpreted as
2-labelled in the defining graph. 

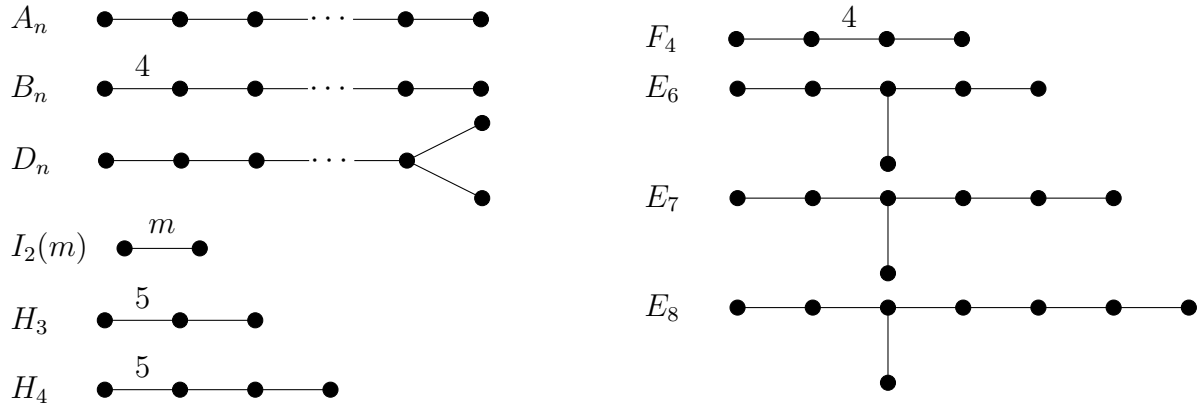
\begin{figure}[H] 

\begin{multicols}{2}
\begin{tikzpicture}
	
	\draw (0,0) -- (2.7,0);
	\draw (3.3, 0) -- (5,0);
	
	\draw [fill] (0,0) circle (.1);
	\draw [fill] (1,0) circle (.1);
	\draw [fill] (2,0) circle (.1);
	\draw [fill] (4,0) circle (.1);
	\draw [fill] (5,0) circle (.1);
	
	\node at (3,0) {$\cdots$};
	\node at (-1,0) {$A_n$};
	 
\end{tikzpicture}

\begin{tikzpicture}

	\draw (0,0) -- (2.7,0);
	\draw (3.3, 0) -- (5,0);
	
	\draw [fill] (0,0) circle(.1);
	\draw [fill] (1,0) circle(.1);
	\draw [fill] (2,0) circle(.1);
	\draw [fill] (4,0) circle(.1);
	\draw [fill] (5,0) circle(.1);
	
	\node at (3,0) {$\cdots$};
	\node at (.5,.3) {4};
	\node at (-1,0) {$B_n$};
	 
\end{tikzpicture}

\begin{tikzpicture}

	\draw (0,0) -- (2.7,0);
	\draw (3.3, 0) -- (4,0);
	\draw (4,0) -- (5,-.5);
	\draw (4,0) -- (5,.5);
	
	\draw [fill] (0,0) circle(.1);
	\draw [fill] (1,0) circle(.1);
	\draw [fill] (2,0) circle(.1);
	\draw [fill] (4,0) circle(.1);
	\draw [fill] (5,-.5) circle(.1);
	\draw [fill] (5,.5) circle(.1);
	
	\node at (3,0) {$\cdots$};
	\node at (-1,0) {$D_n$};
	 
\end{tikzpicture}

\begin{tikzpicture}

	\draw (0,0) -- (1,0);
	
	\draw [fill] (0,0) circle(.1);
	\draw [fill] (1,0) circle(.1);
	
	\node at (.5, .3) {$m$};
	\node at (-1,0) {$I_2(m)$};
	 
\end{tikzpicture}

\begin{tikzpicture}

	\draw (0,0) -- (2,0);
	
	\draw [fill] (0,0) circle(.1);
	\draw [fill] (1,0) circle(.1);
	\draw [fill] (2,0) circle(.1);
	
	\node at (.5, .3) {5};
	\node at (-1,0) {$H_3$};
	 
\end{tikzpicture}

\begin{tikzpicture}

	\draw (0,0) -- (3,0);
	
	\draw [fill] (0,0) circle(.1);
	\draw [fill] (1,0) circle(.1);
	\draw [fill] (2,0) circle(.1);
	\draw [fill] (3,0) circle(.1);
	
	\node at (.5, .3) {5};
	\node at (-1,0) {$H_4$};
	 
\end{tikzpicture}

\begin{tikzpicture}

	\draw (0,0) -- (3,0);
	
	\draw [fill] (0,0) circle(.1);
	\draw [fill] (1,0) circle(.1);
	\draw [fill] (2,0) circle(.1);
	\draw [fill] (3,0) circle(.1);
	
	\node at (1.5, .3) {4};
	\node at (-1,0) {$F_4$};
	 
\end{tikzpicture}

\begin{tikzpicture}

	\draw (0,0) -- (4,0);
	\draw (2,0) -- (2,-1);
	
	\draw [fill] (0,0) circle(.1);
	\draw [fill] (1,0) circle(.1);
	\draw [fill] (2,0) circle(.1);
	\draw [fill] (2,-1) circle(.1);
	\draw [fill] (3,0) circle(.1);
	\draw [fill] (4,0) circle(.1);
	
	\node at (-1,0) {$E_6$};
	 
\end{tikzpicture}

\begin{tikzpicture}

	\draw (0,0) -- (5,0);
	\draw (2,0) -- (2,-1);
	
	\draw [fill] (0,0) circle(.1);
	\draw [fill] (1,0) circle(.1);
	\draw [fill] (2,0) circle(.1);
	\draw [fill] (2,-1) circle(.1);
	\draw [fill] (3,0) circle(.1);
	\draw [fill] (4,0) circle(.1);
	\draw [fill] (5,0) circle(.1);
	
	\node at (-1,0) {$E_7$};
	 
\end{tikzpicture}

\begin{tikzpicture}

	\draw (0,0) -- (6,0);
	\draw (2,0) -- (2,-1);
	
	\draw [fill] (0,0) circle(.1);
	\draw [fill] (1,0) circle(.1);
	\draw [fill] (2,0) circle(.1);
	\draw [fill] (2,-1) circle(.1);
	\draw [fill] (3,0) circle(.1);
	\draw [fill] (4,0) circle(.1);
	\draw [fill] (5,0) circle(.1);
	\draw [fill] (6,0) circle(.1);
	
	\node at (-1,0) {$E_8$};
	 
\end{tikzpicture}

\end{multicols}

\caption{Coxeter diagrams for irreducible finite Coxeter groups}
\label{fig:finitetable}

\end{figure}

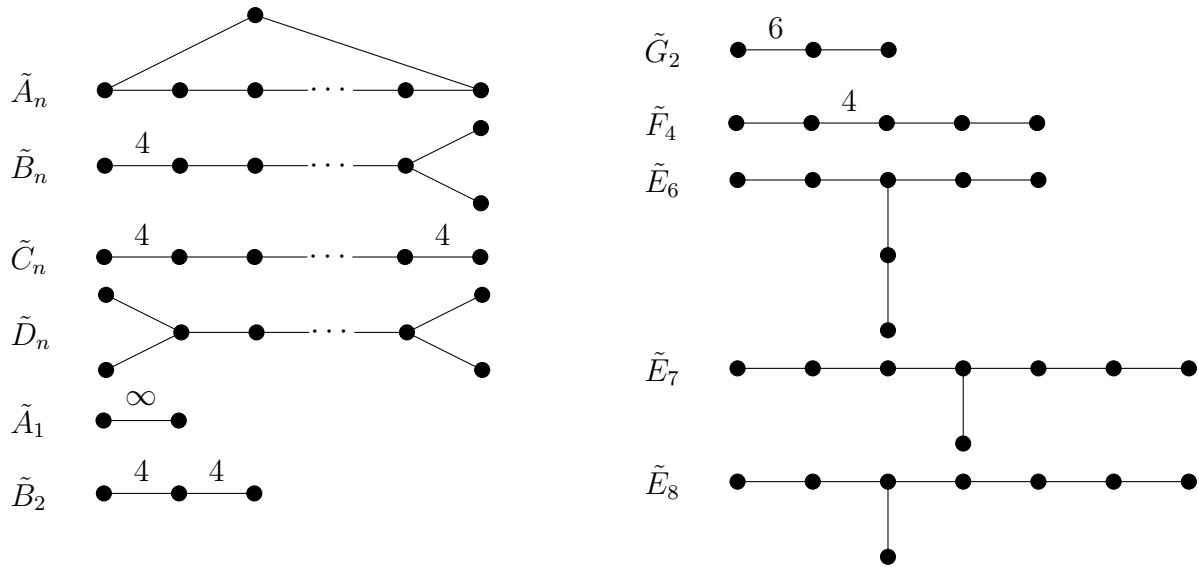
\begin{figure}[H]

\begin{multicols}{2}

\begin{tikzpicture}
	
	\draw (0,0) -- (2.7,0);
	\draw (3.3, 0) -- (5,0);
	\draw (0,0) -- (2,1);
	\draw (2,1) -- (5,0);
	
	\draw [fill] (0,0) circle (.1);
	\draw [fill] (1,0) circle (.1);
	\draw [fill] (2,0) circle (.1);
	\draw [fill] (4,0) circle (.1);
	\draw [fill] (5,0) circle (.1);
	\draw [fill] (2,1) circle (.1);
	
	\node at (3,0) {$\cdots$};
	\node at (-1,0) {$\tilde{A}_n$};
	 
\end{tikzpicture}

\begin{tikzpicture}

	\draw (0,0) -- (2.7,0);
	\draw (3.3, 0) -- (4,0);
	\draw (4,0) -- (5,-.5);
	\draw (4,0) -- (5,.5);
	
	\draw [fill] (0,0) circle(.1);
	\draw [fill] (1,0) circle(.1);
	\draw [fill] (2,0) circle(.1);
	\draw [fill] (4,0) circle(.1);
	\draw [fill] (5,-.5) circle(.1);
	\draw [fill] (5,.5) circle(.1);
	
	\node at (3,0) {$\cdots$};
	\node at (.5,.3) {$4$};
	\node at (-1,0) {$\tilde{B}_n$};
	 
\end{tikzpicture}

\begin{tikzpicture}

	\draw (0,0) -- (2.7,0);
	\draw (3.3, 0) -- (5,0);
	
	\draw [fill] (0,0) circle(.1);
	\draw [fill] (1,0) circle(.1);
	\draw [fill] (2,0) circle(.1);
	\draw [fill] (4,0) circle(.1);
	\draw [fill] (5,0) circle(.1);
	
	\node at (3,0) {$\cdots$};
	\node at (4.5,.3) {4};
	\node at (.5,.3) {4};
	\node at (-1,0) {$\tilde{C}_n$};
	 
\end{tikzpicture}

\begin{tikzpicture}

	\draw (1,0) -- (2.7,0);
	\draw (1,0) -- (0,-.5);
	\draw (1,0) -- (0,.5);
	\draw (3.3, 0) -- (4,0);
	\draw (4,0) -- (5,-.5);
	\draw (4,0) -- (5,.5);
	
	\draw [fill] (0,-.5) circle(.1);
	\draw [fill] (0,.5) circle(.1);
	\draw [fill] (1,0) circle(.1);
	\draw [fill] (2,0) circle(.1);
	\draw [fill] (4,0) circle(.1);
	\draw [fill] (5,-.5) circle(.1);
	\draw [fill] (5,.5) circle(.1);
	
	\node at (3,0) {$\cdots$};
	\node at (-1,0) {$\tilde{D}_n$};
	 
\end{tikzpicture}

\begin{tikzpicture}

	\draw (0,0) -- (1,0);
	
	\draw [fill] (0,0) circle(.1);
	\draw [fill] (1,0) circle(.1);
	
	\node at (.5, .3) {$\infty$};
	\node at (-1,0) {$\tilde{A}_1$};
	 
\end{tikzpicture}

\begin{tikzpicture}

	\draw (0,0) -- (2,0);
	
	\draw [fill] (0,0) circle(.1);
	\draw [fill] (1,0) circle(.1);
	\draw [fill] (2,0) circle(.1);
	
	\node at (.5, .3) {4};
	\node at (1.5,.3) {4};
	\node at (-1,0) {$\tilde{B}_2$};
	 
\end{tikzpicture}

\begin{tikzpicture}

	\draw (0,0) -- (2,0);
	
	\draw [fill] (0,0) circle(.1);
	\draw [fill] (1,0) circle(.1);
	\draw [fill] (2,0) circle(.1);
	
	\node at (.5, .3) {6};
	\node at (-1,0) {$\tilde{G}_2$};
	 
\end{tikzpicture}

\begin{tikzpicture}

	\draw (0,0) -- (4,0);
	
	\draw [fill] (0,0) circle(.1);
	\draw [fill] (1,0) circle(.1);
	\draw [fill] (2,0) circle(.1);
	\draw [fill] (3,0) circle(.1);
	\draw [fill] (4,0) circle(.1);
	
	\node at (1.5, .3) {4};
	\node at (-1,0) {$\tilde{F}_4$};
	 
\end{tikzpicture}

\begin{tikzpicture}

	\draw (0,0) -- (4,0);
	\draw (2,0) -- (2,-2);
	
	\draw [fill] (0,0) circle(.1);
	\draw [fill] (1,0) circle(.1);
	\draw [fill] (2,0) circle(.1);
	\draw [fill] (3,0) circle(.1);
	\draw [fill] (4,0) circle(.1);
	\draw [fill] (2,-1) circle(.1);
	\draw [fill] (2,-2) circle(.1);
	
	\node at (-1,0) {$\tilde{E}_6$};
	 
\end{tikzpicture}

\begin{tikzpicture}

	\draw (0,0) -- (6,0);
	\draw (3,0) -- (3,-1);
	
	\draw [fill] (0,0) circle(.1);
	\draw [fill] (1,0) circle(.1);
	\draw [fill] (2,0) circle(.1);
	\draw [fill] (3,0) circle(.1);
	\draw [fill] (4,0) circle(.1);
	\draw [fill] (5,0) circle(.1);
	\draw [fill] (6,0) circle(.1);
	\draw [fill] (3,-1) circle(.1);	

	\node at (-1,0) {$\tilde{E}_7$};

\end{tikzpicture}

\begin{tikzpicture}

	\draw (0,0) -- (6,0);
	\draw (2,0) -- (2,-1);
	
	\draw [fill] (0,0) circle(.1);
	\draw [fill] (1,0) circle(.1);
	\draw [fill] (2,0) circle(.1);
	\draw [fill] (3,0) circle(.1);
	\draw [fill] (4,0) circle(.1);
	\draw [fill] (5,0) circle(.1);
	\draw [fill] (6,0) circle(.1);
	\draw [fill] (2,-1) circle(.1);	

	\node at (-1,0) {$\tilde{E}_8$};

\end{tikzpicture}

\end{multicols}

\caption{Coxeter diagrams for Euclidean reflection groups}
\label{fig:affinetable}

\end{figure}

\bibliographystyle{plain}
\bibliography{references}

\end{document}